\documentclass[11pt,oneside]{amsart}   	
\allowdisplaybreaks
\usepackage{geometry} 

\usepackage{longtable}
\usepackage{color}
\usepackage{colortbl}
\definecolor{TableGray}{gray}{0.95}

\usepackage{amsmath, amssymb}
\usepackage[unicode=true,pdfusetitle,
 bookmarks=true,bookmarksnumbered=false,bookmarksopen=false,
 breaklinks=false,pdfborder={0 0 0},pdfborderstyle={},backref=false,colorlinks=true]
 {hyperref}
\hypersetup{
 pdfencoding=auto, psdextra}

\usepackage[ruled,linesnumbered]{algorithm2e}
\usepackage{siunitx}

\newtheorem{theorem}{\bf Theorem}[section]
\newtheorem{proposition}[theorem]{\bf Proposition}
\newtheorem{lemma}[theorem]{\bf Lemma}
\newtheorem{corollary}[theorem]{\bf Corollary}

\theoremstyle{definition}
\newtheorem{example}[theorem]{\bf Example}
\newtheorem{definition}[theorem]{\bf Definition}
\newtheorem{remark}[theorem]{\bf Remark}

\DeclareMathOperator{\supp}{supp}

\DeclareMathOperator{\Aut}{Aut}
\DeclareMathOperator{\Orb}{Orb}

\newcommand{\bdot}{\boldsymbol{\cdot}}
\def\length{\ell}
\def\cay{\mathrm{Cay}}
\def\diam{\mathrm{diam}}
\def\d#1{\mathsf{d}(#1)}
\def\dstar#1{\mathsf{d}^*(#1)}
\def\D#1{\mathsf{D}(#1)}
\def\GD#1{\mathsf{GD}(#1)}
\def\ddiam#1{\mathsf{dcd}(#1)}
\def\ddiamstar#1{\mathsf{dcd}^*(#1)}

\title{The directed Cayley diameter and the Davenport constant }

\author[R\'eka Andr\'as]{R\'eka Andr\'as}
\address{E\"otv\"os Lor\'and University, Faculty of Informatics, H-1117 Budapest, Pázmány P. sny 1/C., Hungary}
\email{andrasreka96@gmail.com}
\author[K\'alm\'an Cziszter]{K\'alm\'an Cziszter}
\address{Budapest Business University, Markó u. 29-31, 1055 Budapest, Hungary, ORCID iD: https://orcid.org/0000-0002-8034-836X}
\email{cziszter.kalman@gmail.com} 
\author[M\'aty\'as Domokos]
{M\'aty\'as Domokos}
\address{HUN-REN Alfr\'ed R\'enyi Institute of Mathematics,
Re\'altanoda utca 13-15, 1053 Budapest, Hungary,
ORCID iD: https://orcid.org/0000-0002-0189-8831}
\email{domokos.matyas@renyi.hu}
\author[Istv\'an Sz\"oll\H osi]{Istv\'an Sz\"oll\H osi} 
\address{Babe\c s-Bolyai University, Faculty of Mathematics and Computer Science, str. M. Kog\u alniceanu, nr. 1, 400084, Cluj-Napoca, Romania, ORCID iD: https://orcid.org/0000-0001-8632-9760} 
\email{szollosi@gmail.com}
\begin{document}
\thanks{Partially supported by the Hungarian National Research, Development and Innovation Office,  NKFIH K 138828,  K 132002.}
\subjclass[2010]{Primary 20D60; Secondary 20M14, 20K01}
\keywords{Cayley graph, directed diameter of a group, monoid of product-one sequences, Davenport constant}

\begin{abstract} 
The directed Cayley diameter of a finite group is investigated in terms of the 
monoid of product-one sequences over the group, via the new notion of directed geodesic atoms. 
Two quantities associated to the set of directed geodesic atoms provide lower and upper bounds for the directed Cayley diameter. 
An algorithm for computing 
the directed geodesic atoms is implemented in GAP, 
and is applied to determine the above mentioned quantities for all non-abelian groups of order at most $42$, and for the alternating group of degree $5$.  
Furthermore, the small and large Davenport constants of all these groups are computed (excepting the large Davenport constant for $A_5$), extending thereby the formerly obtained results on the groups of order less than $32$. 
Along the way the directed Cayley diameter of a finite abelian group is 
expressed in terms of its invariants. 
\end{abstract} 

\maketitle

\section{Introduction} \label{sec:intro} 

It was proved recently in \cite{geroldinger-oh} that if the monoid of product-one sequences 
over a finite group $G_1$ is isomorphic to the monoid of product-one sequences over a finite group $G_2$, then the group $G_1$ is isomorphic to the group $G_2$. 
It is therefore not surprising that much information on a finite group $G$ can be 
read off from the set of atoms in the monoid of product-one sequences over 
$G$. The present paper gives further instances of this principle. We shall 
relate the problem of computing the directed Cayley diameter of a finite group to the study of 
the atoms in the monoid of product-one sequences 
over the group. 
Note that the directed Cayley diameter is an obvious upper bound for the undirected Cayley diameter (referred to simply as the ``Cayley diameter''). Conversely, a general explicit bound for the directed Cayley diameter in terms of the Cayley diameter is proved in \cite{babai}. 
The Cayley diameter of finite groups has received much attention in the literature. In particular, several prominent papers are motivated by the conjecture from \cite{babai-seress} stating 
that the Cayley diameter of a finite non-abelian simple group $G$ is bounded above by 
$(\log|G|)^c$ with some absolute constant 
$c$, see for example \cite{helfgott-seress}, 
\cite{breuillard-green-guralnick-tao}, \cite{pyber-szabo}, \cite{halasi-maroti-pyber-qiao}, 
\cite{eberhard-jezernik}. 
Other areas where Cayley graphs have relevance are surveyed in \cite{konstantinova}.  

Recall that the directed Cayley diameter of a finite group $G$ is the maximal diameter of its Cayley graphs associated with the various irredundant generating systems of $G$. 
The starting point of this paper is the observation \cite[Proposition 5.6]{cz-d-sz}, asserting that the large Davenport constant $\D{G}$ of a finite group $G$ (i.e. the maximal length of an atom in 
the monoid of product-one sequences over $G$) is strictly greater than the diameter of any of the Cayley digraphs of $G$. 
In fact this statement can be strengthened by saying that $\D{G}$ is strictly greater than the directed Cayley diameter of any subgroup of $G$ (see 
Corollary~\ref{cor:D>max diamCay(H)}). 
By our Lemma~\ref{lemma:cayley diam from geodesic atoms}, the maximum of the directed Cayley diameters of the subgroups of $G$ is one less than $\GD{G}$, the \emph{directed geodesic Davenport constant of} $G$ 
introduced in Definition~\ref{def:GD(G)} as the maximal length of a so-called \emph{directed geodesic atom} (see Definition~\ref{def:geodesic atom}) in the monoid 
$\mathcal{B}(G)$ of product-one sequences over $G$. 
A characterisation of the directed geodesic atoms, 
given in Section~\ref{sec:geodesic atoms}, 
together with Theorem~\ref{thm:equivalent}, offers an algorithm to build up the set of directed geodesic atoms, using a bottom-up approach (see Section~\ref{sec:algorithm}). 
This algorithm is related to the algorithm given in \cite[Section 6]{cz-d-sz}, computing $\D{G}$ and all atoms in 
$\mathcal{B}(G)$. Finally, once a list of directed geodesic atoms is produced, one can read off bounds for 
the directed Cayley diameter of $G$ 
(see Lemma~\ref{lemma:cayley diam from geodesic atoms}). 

After recalling concepts related to  Cayley graphs and the monoid of product-one sequences in Section~\ref{sec:prel}, 
we turn to finite abelian groups and determine 
their directed Cayley diameter in 
Theorem~\ref{thm:diam=d*}. 
The materials on the relation between the directed Cayley diameter and the atoms in the monoid of product-one sequences as well as the algorithms are contained in  
Sections~\ref{sec:Davenport-Cayley}, \ref{sec:geodesic atoms}, \ref{sec:characterization}, and \ref{sec:algorithm}. 
The results of our computer calculations 
are presented in Section~\ref{sec:calculations}. 
The small and large Davenport constants of the non-abelian groups of order less than $32$ were computed in \cite{cz-d-sz}. Here 
we give the Davenport constants of the remaining non-abelian groups of order at most $42$ as well, and for all non-abelian 
groups $G$ with $|G|\le 42$ we compute  also the directed geodesic Davenport constant of $G$ (which in several cases yields 
the directed Cayley diameter of $G$). 

Finally, in Section~\ref{sec:survey} we review some recent results on the Davenport constants of finite non-abelian 
groups. 

We note that in a follow-up paper we shall adapt the results of this paper to  the problem of computing the (undirected) Cayley diameter of a finite group.

\section{Preliminaries}\label{sec:prel} 

\subsection{Directed Cayley graphs}
Let $G$ be a finite group (written multiplicatively unless explicitly stated otherwise) with identity element $1_G$ and let $B\subseteq G$ be a generating system of $G$.  
The \emph{directed Cayley graph} $\cay(G,B)$ of $G$ with respect to the generating system $B$ 
is the digraph whose vertices are labeled by the elements of $G$, and 
the pair $(g,h)$ is a directed edge in $\cay(G,B)$ if there exists an element 
$b\in B$ with $h=gb$. 
We say that a directed path $g_0,g_1,\dots,g_d$ in $\cay(G,B)$ is
\emph{geodesic} if its length $d$ is minimal among the lengths of the directed paths in $\cay(G,B)$ from $g_0$ to $g_d$. 
The \emph{diameter} $\diam(\cay(G,B))$ of $\cay(G,B)$ is 
the maximal possible length of a geodesic path in $\cay(G,B)$. As the action of $G$ on itself via left multiplication embeds $G$ into the automorphism group of the digraph $\cay(G,B)$, $\diam(\cay(G,B))$ is the same as the maximal length of a geodesic path in $\cay(G,B)$ starting at $1_G$. 
The \emph{directed Cayley diameter} of $G$ is defined as 
\[\ddiam{G}:=\max\{\diam(\cay(G,B))\mid B\subseteq G\mbox{ generates } G\}.\] 
The number $\ddiam{G}$ can be introduced also without a reference to Cayley graphs. Namely, for a subset $B\subseteq G$ and an element $g\in \langle B\rangle$ 
(the subgroup of $G$ generated by $B$) 
we set 
\[\length_B(g):=\min\{n\in\mathbb{Z}_{\ge 0}\mid \exists 
b:\{1,\dots,n\}\to B \text{ with }g=b(1)\cdots b(n)\}\] 
(with the convention $\length_B(1_G)=0$). 
Clearly, 
 a directed path $1_G=g_0,g_1,\dots,g_d=g$ from $1_G$ to $g$ in $\cay(G,B)$ is \emph{geodesic} if 
and only if $d=\length_B(g)$. 
Therefore 
\begin{equation}\label{eq:diam(cay(G,B))}\diam(\cay(G,B))=\max\{\length_B(g) \mid g\in G \},\end{equation} 
and  
\begin{equation}\label{eq:ddim-length} \ddiam{G}=\max\{\length_B(g) \mid g\in G, \  B\subseteq G,\  \langle B\rangle=G \}.\end{equation}

\subsection{The monoid of product-one sequences}

Denote by $(\mathcal{F}(G),\bdot)$ the free abelian monoid 
with basis $G$. Its element will be called a \emph{sequence over $G$}, because 
the element 
\[S=g_1\bdot\dots\bdot g_\ell\in\mathcal{F}(G)\] 
can be thought of as the sequence  $g_1,\dots,g_\ell$, where repetition is allowed and the order of terms is disregarded. The number $\ell$ is called the \emph{length} of $S$, denoted by $|S|$. 
Note that the product of $g_1\bdot \dots \bdot g_\ell$ and $h_1\bdot\dots\bdot h_k$ in 
the monoid $\mathcal{F}(G)$ is $g_1\bdot \dots \bdot g_\ell\bdot h_1\bdot\dots\bdot h_k$ 
(concatenation of sequences). 
For $S=g_1\bdot\dots\bdot g_\ell\in\mathcal{F}(G)$ we shall use the notation 
\[\mathsf{v}_g(S):=|\{i\in\{1,\dots,\ell\}\mid g_i=g\}|.\] 
For $g\in G$ and a positive integer $n$ we shall write $g^{[n]}$ for the sequence 
$\underbrace{g\bdot \dots\bdot g}_n$. 
Define the \emph{support} of $S\in \mathcal{F}(G)$ as 
\[\mathrm{supp}(S):=\{g\in G\mid  \mathsf{v}_g(S)>0\}.\] 
For $g\in \mathrm{supp}(S)$ denote by $S\bdot g^{[-1]}$ the element of 
$\mathcal{F}(G)$ defined by the equality
\[S=g\bdot (S\bdot g^{[-1]}).\]  

The sequence $g_1\bdot \dots \bdot g_\ell$ is a \emph{product-one sequence} if 
there exists a permutation $\pi\in \mathrm{Sym}_\ell$ such that 
$g_{\pi(1)}\cdots g_{\pi(\ell)}=1_G\in G$. Product-one sequences form a submonoid 
$\mathcal{B}(G)$ of $\mathcal{F}(G)$. 
An element $S\in \mathcal{B}(G)$ is called an \emph{atom} if it can not be written 
as $S=T\bdot U\in \mathcal{F}(G)$ with $T,U\in \mathcal{B}(G)$ and $|T|>0$, $|U|>0$.  Write $\mathcal{A}(G)$ for the set of atoms in $\mathcal{B}(G)$. 
The maximal length of an atom in the monoid $\mathcal{B}(G)$ is denoted by $\D{G}$ and is called the 
\emph{large Davenport constant of $G$}. 
A sequence $S=g_1\bdot \dots \bdot g_\ell\in \mathcal{F}(G)$ is \emph{product-one free} if $S$ can not be written as $S=T\bdot U\in \mathcal{F}(G)$ where $T\in \mathcal{B}(G)$ and $|T|>0$. 
The maximal length of a product-one free sequence is denoted by $\d{G}$ and is 
called the \emph{small Davenport constant of $G$}. 
The small and large Davenport constants are related by the following inequality: 
$\D{G}\ge 1+\d{G}$. 
 
\section{Abelian groups} \label{sec:abelian} 

When $G$ is abelian, we shall use the additive notation, and write $0_G$ for the zero element of $G$. A product-one sequence 
$S=g_1\bdot \dots \bdot g_\ell$ 
in this case is rather called  a \emph{zero-sum sequence}. It means that  
$g_1+\cdots+g_\ell=0_G\in G$.  
A product-one free sequence is called a zero-sum free sequence in this case. 
So $S=g_1\bdot \dots \bdot g_\ell$ is \emph{zero-sum free} if there is no 
non-empty subset $\{i_1,\dots,i_k\}\subseteq\{1,\dots,\ell\}$ with 
$g_{i_1}+\cdots+g_{i_k}=0_G\in G$. 
The small and large Davenport constants of a finite abelian group $G$ 
are related by the equality $\D{G}=\d{G}+1$. 

\begin{proposition}\label{prop:d>=diam}
For a finite abelian group $G$ we have 
$\d{G}\ge \ddiam{G}$. 
\end{proposition} 

\begin{proof}
Write $d:=\ddiam{G}$. If $d=0$, then 
$G$ is the one-element group, and 
both $\d{G}$ and $\ddiam{G}$ are equal to zero. From now on we assume that 
$d>0$. 
By \eqref{eq:ddim-length} there exists an equality $g=g_1+\cdots+g_d$, with 
$\{g_1,\dots,g_d\}\subseteq B$, where $B$ is an irredundant generating system of $G$ 
(i.e. $B$ generates $G$, but no proper subset of $B$ generates $G$), and 
$d=\length_B(g)$. Then $g\neq 0_G$, 
and $g_1\bdot \dots \bdot g_d$ is a zero-sum free sequence, and so $\d{G}\ge d$. 
Indeed, otherwise after a possible renumbering we have that for some $k<d$ we have $g_{k+1}+\cdots+g_d=0_G$, implying that $g_1+\cdots+g_k=g$. Thus 
$\length_B(g)\le k<d$, contrary to the assumption that $\length_B(g)=d$. 
\end{proof}

For an abelian group $G=C_{n_1}\oplus \cdots\oplus C_{n_r}$  with 
$2\le n_1\mid\dots\mid n_r$, set 
\[\dstar{G}:=\sum_{i=1}^r(n_i-1).\] 
By  \cite[Lemma 4.1]{Gr-Ma-Or09} (see also \cite[Exercise 1.6]{Gr13a}) we have 
\begin{equation}\label{eq:d*(G)}
    \mathsf d^*(G)\ge \mathsf d^*(H)+\mathsf d^*(G/H) \text{ for any subgroup }H \text{ of }G. 
\end{equation}
In particular, 
\begin{equation}\label{eq:d*(G/H)}
\mathsf d^*(H)+|G:H|-1\le \dstar{G} \quad \mbox{ when }G/H \mbox{ is cyclic}
\end{equation}
(see also \cite[Lemma 3.2]{domokos} for an alternative proof).

Note that $\d{G}\ge \dstar{G}$ for all $G$. 
Therefore Proposition~\ref{prop:d>=diam} follows from the sharper  
Theorem~\ref{thm:diam=d*} below. 

\begin{theorem}\label{thm:diam=d*} 
For any finite abelian group $G$ we have the equality 
\[\ddiam{G}=\dstar{G}.\]  
\end{theorem} 

\begin{proof} 
For $i=1,\dots,r$ denote by $e_i$ the generator of the $i$th direct summand $C_{n_i}$ of $G=C_{n_1}\oplus \cdots \oplus C_{n_r}$, where $2\le n_1\mid\dots\mid n_r$. Consider the generating system $B:=\{e_1,\dots,e_r\}$ of $G$. 
Then $\length_B(\sum_{i=1}^r(n_i-1)e_i)=\sum_{i=1}^r(n_i-1)$, implying 
the inequality $\ddiam{G}\ge \sum_{i=1}^r(n_i-1)=\dstar{G}$. 

We prove the reverse inequality $\ddiam{G}\le \dstar{G}$ by induction on $|G|$. 
For $|G|=1,2,3$ the statement is obvious. 
Suppose $|G|\ge 4$, and take a generating system $B$ of $G$ such that 
$d:=\ddiam{G}=\diam(\cay(G,B))$. 
For any subset $C\subseteq B$ 
it is obvious that  
$\diam(\cay(G,C))\ge \diam(\cay(G,B))$, therefore by omitting redundant elements from $B$, 
we may assume that $B$ is irredundant. 
Pick some $b\in B$. Then $B\setminus \{b\}$ generates a proper subgroup $H$ of 
$G$, therefore $\ddiam{H}\le \dstar{H}$ by the induction hypothesis.  
Now take an element $g\in G$ with $\length_B(g)=d$. 
Denote by $\bar g$ the image of $g$ in the factor group $G/H$. The factor group 
$G/H$ is generated by the image $\bar b$ of $b$. So there exists a 
$k\in \{0,1,\dots,|G:H|-1\}$ such that $\bar g =k\bar b$, or equivalently, 
$g-kb\in H$. 
Now 
\[\length_{B\setminus \{b\}}(g-kb)\le\diam(\cay(H,B\setminus \{b\}))
\le \ddiam{H}\le \dstar{H}, \]
hence 
\[d=\length_B(g)
\le \length_{B\setminus \{b\}}(g-kb)+k
\le \dstar{H}+k
\le  \dstar{H}+|G:H|-1.\] 
Finally, by \eqref{eq:d*(G/H)} we conclude 
$\ddiam{G}=d\le \dstar{G}$. 
\end{proof} 

\begin{remark} \begin{itemize} 
\item[(i)] An easy short argument yielding an upper bound for 
$\ddiam{G}$ in the case when $G$ is the direct sum of prime order groups is given in \cite[Lemma 5.2]{babai-seress:2}. 
\item[(ii)] The undirected Cayley diameter of a finite abelian group is determined in \cite[Lemma 3.3]{wilson}. For  $G=C_{n_1}\oplus \cdots\oplus C_{n_r}$  with 
$2\le n_1\mid\dots\mid n_r$, it is 
$\sum_{i=1}^r\lfloor n_i/2\rfloor$. 
\end{itemize} 
\end{remark} 

It is well known that the equality $\d{G}=\dstar{G}$ holds for 
abelian groups of rank at most $2$, for abelian $p$-groups, and for several other 
infinite classes of abelian groups, and this equality is conjectured to hold for 
abelian groups of rank $3$ and for the groups of the form $C_n\oplus\cdots\oplus C_n$.  
Therefore  
for a notable class of abelian groups the inequality in the statement of 
Proposition~\ref{prop:d>=diam} holds with equality.

\begin{example} \label{example:d>d*} 
Consider the group $G:=C_3\oplus C_3\oplus C_3\oplus C_6$ with generators 
$e_i$ $(i=1,\dots,4)$ such that $\mathrm{ord}(e_1)=\mathrm{ord}(e_2)=\mathrm{ord}(e_3)=3$ and $\mathrm{ord}(e_4)=6$.  
Set $g_1:=-e_1+e_2+e_3+e_4$, 
$g_2:=e_1-e_2+e_3+e_4$, 
$g_3:=e_1+e_2-e_3+e_4$, 
$g_4:=-e_1+e_2+e_3-e_4$, 
$g_5=e_1+e_2+e_3+e_4$, 
$g_6:=e_2-e_3-e_4$, 
$g_7:=-e_2+e_3-e_4$. 
It is shown (as a special case of a more general statement) in 
\cite{geroldinger-schneider} that the sequence 
$g_1^{[2]}\bdot g_2^{[2]}\bdot g_3^{[2]}\bdot g_4^{[2]}\bdot g_5^{[2]}\bdot g_6 \bdot g_7$ is zero-sum free. 
As this sequence has length $12$, we conclude that 
$\d{G}\ge 12 > 11=\dstar{G}$. 
 Any ordering of the elements in this sequence gives a path of length $12$ from $0$ to 
 $g:=2(g_1+g_2+g_3+g_4+g_5)+g_6+g_7$ in the Cayley graph 
 $\cay(G,\{g_i\mid i=1,\dots,7\})$. 
 Such a path can not be a geodesic path from $0$ to $g$ in $\cay(G,\{g_i\mid i=1,\dots,7\})$ 
 by Theorem~\ref{thm:diam=d*}.  
 Indeed, the equality 
 \[2(g_1+g_2+g_3+g_4+g_5)+g_6+g_7=4g_1+g_2+g_3+2g_4,\] 
 shows that there exists a path of length $4+1+1+2=8$ from $0_G$ to $g$ 
 in  $\cay(G,\{g_i\mid i=1,\dots,7\})$, hence $\length_{\{g_i\mid i=1,\dots,7\}}(g)\le 8$. 
 \end{example}

\section{The relation between the Davenport constant and the directed Cayley diameter} \label{sec:Davenport-Cayley} 

In view of Proposition~\ref{prop:d>=diam} it is natural to pose the following question: does the inequality $\ddiam{G}\le \d{G}$ hold for all non-abelian finite groups $G$?

The computer calculations presented in Section~\ref{sec:calculations} answer 
this question in the negative, see for example the group with GAP-identifier \texttt{[16,3]} in the table in 
Section~\ref{sec:calculations}.  
What we know instead is that the directed Cayley diameter can be bounded in terms of the 
large Davenport constant, since 
in \cite[Proposition 5.6]{cz-d-sz} the following inequality was pointed out: 

\begin{proposition}\label{prop:D>diamCay} 
We have $\D{G}>\ddiam{G}$ 
for all finite groups $G$. 
\end{proposition} 

Proposition~\ref{prop:D>diamCay} follows from the following observation: 

\begin{lemma}\label{lemma:geodesic atoms} Suppose that 
$\length_{\supp(b_1\bdot \dots \bdot b_d)}(b_1\cdots b_d)=d$ for a sequence $b_1\bdot \dots \bdot b_d\in \mathcal{F}(G)$. 
Then setting $b_{d+1}:=(b_1\cdots b_d)^{-1}$, we have that $b_1\bdot \dots \bdot b_d\bdot b_{d+1}$  is an atom in $\mathcal{B}(G)$.   
\end{lemma} 

\begin{proof} Suppose for contradiction  that 
$b_1\bdot \dots \bdot b_d\bdot b_{d+1}=S\bdot T$ for some $S,T\in \mathcal{B}(G)$ with $|S|>0$ and $|T|>0$. 
By symmetry we may assume that $\mathsf{v}_{b_{d+1}}(T)\ge 1$, 
so $T=c_1\bdot \dots\bdot c_m\bdot b_{d+1}$ for some $c_1,\dots,c_m\in \supp(b_1\bdot \dots \bdot b_d)$. 
Moreover, after a possible renumbering we have $c_1\cdots c_mb_{d+1}=1_G$.  
Thus $c_1\cdots c_m=b_{d+1}^{-1}=b_1\cdots b_d$, hence 
$\length_{\supp(b_1\bdot\dots\bdot b_d)}(b_{d+1}^{-1})\le 
\length_{\supp(c_1\bdot \dots \bdot c_m)}(b_{d+1}^{-1})\le m<d$. 
This contradicts the assumption 
$\length_{\supp(b_1\bdot \dots \bdot b_d)}(b_1\cdots b_d)=d$. 
\end{proof} 

Lemma~\ref{lemma:geodesic atoms} implies the following strengthening of Proposition~\ref{prop:D>diamCay}: 

\begin{corollary} 
\label{cor:D>max diamCay(H)} 
For any finite group $G$ we have the inequality  
\begin{equation}\label{eq:D>max cdc(H)}
    \D{G}>\max\{\ddiam{H}\mid H\le G\}.
    \end{equation}
\end{corollary}

\begin{remark} \label{remark:diam(H)>diam(G)}
Corollary~\ref{cor:D>max diamCay(H)} is stronger than 
Proposition~\ref{prop:D>diamCay}, since 
in some cases it happens that $\ddiam{H}>\ddiam{G}$ for a subgroup $H$ of $G$ (see Remark~\ref{remark:not monotone} and Remark~\ref{remark:monotonity}). 
\end{remark}

\section{Directed geodesic atoms and their applications}\label{sec:geodesic atoms}

The number on the right hand side of the inequality \eqref{eq:D>max cdc(H)} can be characterized in terms of the lengths 
of some special atoms in $\mathcal{B}(G)$.

\begin{definition} \label{def:geodesic atom}
An element $S\in \mathcal{B}(G)$ 
is called a \emph{directed geodesic atom} if 
there exists a $g\in \supp(S)$ such that 
$\length_{\supp(S\bdot g^{[-1]})}(g^{-1})=|S|-1$. 

Write 
$\mathcal{GA}(G)$ for the set of directed geodesic atoms in $\mathcal{B}(G)$. 
\end{definition}

\begin{remark} The terminology ``directed geodesic atom'' is motivated by the following: 
\begin{itemize}
\item[(i)]  
$S\in \mathcal{B}(G)$ is a directed geodesic atom in the sense of Definition~\ref{def:geodesic atom} if 
and only if $S$ can be written as 
$b_1\bdot \dots\bdot b_d\bdot b_{d+1}$, 
where $1_G,b_1,b_1b_2,\dots,b_1\cdots  b_d$ 
is a geodesic path from $1_G$ to $b_1\cdots  b_d$ in 
the directed Cayley graph $\cay(\langle b_1,\dots,b_d\rangle, \supp(b_1\bdot \dots \bdot b_d))$, and $b_{d+1}^{-1}=b_1\cdots b_d$. 
\item[(ii)] By Lemma~\ref{lemma:geodesic atoms} a directed geodesic atom in $\mathcal{B}(G)$ is an atom in $\mathcal{B}(G)$, 
so 
\begin{equation}\label{eq:GA subset A}
\mathcal{GA}(G)\subseteq \mathcal{A}(G).\end{equation}
\end{itemize}
\end{remark} 

 \begin{example} \label{example:non-geodesic atom} 
 Keeping the notation of  Example~\ref{example:d>d*}, set 
 $g_8:=-(2g_1+2g_2+2g_3+2g_4+2g_5+g_6+g_7)$.  
 The sequence $g_1^{[2]}\bdot g_2^{[2]}\bdot g_3^{[2]}\bdot g_4^{[2]}\bdot g_5^{[2]}\bdot g_6 \bdot g_7\bdot g_8$ is an atom in $\mathcal{B}(C_3\oplus C_3\oplus C_3\oplus C_6)$, 
 but our considerations in Section~\ref{sec:abelian} show that it is not a directed geodesic atom.  
 \end{example} 

In analogy with the large Davenport constant, we define the following quantities: 

\begin{definition}\label{def:GD(G)}
\begin{itemize} 
\item[(i)] The \emph{geodesic large Davenport constant of} $G$ is 
\[\GD{G}:=\max\{|S|\mid S\in \mathcal{GA}(G)\}.\]
\item[(ii)] Set 
\[\ddiamstar{G}:=\max\{|S|\colon S\in \mathcal{GA}(G),\ \langle \supp(S)\rangle=G\}-1.\]
\end{itemize}
    \end{definition} 

As an immediate consequence of 
\eqref{eq:ddim-length} and Definition~\ref{def:geodesic atom} we 
get the following: 

\begin{lemma}\label{lemma:cayley diam from geodesic atoms}
We have 
\begin{itemize}
\item[(i)] $\ddiam{G}\ge \ddiamstar{G}$.
\item[(ii)] $\max\{\ddiam{H}\mid H\le G\}=\GD{G}-1$. 
\end{itemize}
\end{lemma}

\begin{corollary} \label{cor:diamCay<GD(G)} 
For any finite group $G$ we have  
\[\ddiamstar{G}\le \ddiam{G}\le \GD{G}-1\le \D{G}-1.\] 
\end{corollary} 
\begin{proof} 
The first two inequalities follow from 
Lemma~\ref{lemma:cayley diam from geodesic atoms}, the third inequality follows from \eqref{eq:GA subset A}.   
\end{proof}

Lemma~\ref{lemma:cayley diam from geodesic atoms} (ii) implies that 
$\ddiam{G}$ and $\GD{G}$ are tightly 
linked for certain groups as follows: 

\begin{corollary} \label{prop:monotone cayley diameter} 
The following conditions are equivalent 
for a finite group $G$: 
\begin{itemize}
    \item[(i)] $\ddiam{G}\ge \ddiam{H}$ for all  
subgroups $H$ of $G$.
\item[(ii)] $\ddiam{G}=\GD{G}-1$. 
\end{itemize}
Moreover, if $\ddiamstar{G}=\GD{G}-1$, then (i) and (ii) hold. 
\end{corollary} 
 
\begin{remark}\label{remark:not monotone}
The equivalent conditions (i) and (ii) of 
Corollary~\ref{prop:monotone cayley diameter} 
hold for all finite abelian groups $G$ by 
Theorem~\ref{thm:diam=d*}. 
However, they do not hold for all non-abelian finite groups. The smallest examples are the groups with 
GAP identifiers 
\texttt{[16,6], [16,8], [16,9]} in the table in Section~\ref{sec:calculations}. 
\end{remark}

Generating systems of $G$ can also be read off from $\mathcal{GA}(G)$ 
(and hence from $\mathcal{A}(G)$): 

\begin{proposition} \label{prop:generators} 
A subset $B$ generates a finite group $G$ if and only if for all $g\in G$ there exists an 
$S\in \mathcal{GA}(G)$ with $g\in \mathrm{supp}(S)$ and 
$\mathrm{supp}(S\bdot g^{[-1]})\subseteq B$. 
\end{proposition} 

\begin{proof}
Suppose that $B$ generates $G$. Then for each $g\in G$ there exist 
$b_1,\dots,b_n\in B$ with $g^{-1}=b_1\cdots b_n$, and $n=\length_B(g^{-1})$. 
Now $S:=b_1\bdot\dots\bdot b_n\bdot g$ belongs to $\mathcal{GA}(G)$ and $S$ satisfies 
$g\in \mathrm{supp}(S)$ and 
$\mathrm{supp}(S\bdot g^{[-1]})\subseteq B$. 

Conversely, let $B$ be a subset of $G$ such that for all $g\in G$ there exists an 
$S\in \mathcal{GA}(G)$ with $g\in \mathrm{supp}(S)$ and 
$\mathrm{supp}(S\bdot g^{[-1]})\subseteq B$. Then in particular, for all $g\in G$ there exist elements $b_1,\dots,b_n\in B$ with $b_1\cdots b_ng=1_G$, i.e. 
$g^{-1}=b_1\cdots b_n$. That is, $B$ generates $G$. 
\end{proof} 

Given the set $\mathcal{GA}(G)$ and a fixed generating system $B$ of $G$, 
it is in principle an easy matter to determine $\diam(\cay(G,B))$. 
Select the following subsets of $\mathcal{GA}(G)$: 
\[\mathcal{GA}(G,B,0):=\{S\in \mathcal{GA}(G)\mid \mathrm{supp}(S)\subseteq B\}.\]  
For $g\in G\setminus B$ set 
\[\mathcal{GA}(G,B,g,1):=
\{S\in \mathcal{GA}(G)\mid 
\mathsf{v}_g(S)=1,
\ \mathrm{supp}(S\bdot g^{[-1]})\subseteq B\},\]
and set 
\[\mathcal{GA}(G,B,g):=
\{S\in   \mathcal{GA}(G,B,g,1)\mid \forall T\in \mathcal{GA}(G,B,g,1): |T|\ge |S|\}.\] 
Finally, 
\[\mathcal{GA}(G,B):=\mathcal{GA}(G,B,0)\cup \bigcup_{g\in G\setminus B}\mathcal{GA}(G,B,g).\] 

\begin{proposition}\label{prop:diamCay(G,B)}
We have the equality 
\begin{equation}\label{eq:max Slg} 
\diam(\cay(G,B))=\max\{|S| \mid S\in \mathcal{GA}(G,B)\}-1.
\end{equation}
\end{proposition} 

\begin{proof} 
Set $d:=\diam(\cay(G,B))$ and write $m$ for the number on the right hand side of \eqref{eq:max Slg}. First we show that $m\ge d$. By definition of $d$ (see \eqref{eq:diam(cay(G,B))}) there exists 
an element $h:=b_1\cdots b_d$, where 
$b_1,\dots,b_d\in B$ and $\length_B(h)=d$. 
Then for $g:=h^{-1}$ we have that 
$b_1\bdot \dots\bdot b_d\bdot g$ is a directed geodesic atom, and it belongs to $\mathcal{GA}(G,B,0)$ if $g\in B$, and it belongs to $\mathcal{GA}(G,B,g)$ if $g\in G\setminus B$. This shows that $m\ge d$. 

By definition of $m$ there exists a directed geodesic atom 
$S=b_1\bdot \dots \bdot b_m\bdot g\in \mathcal{GA}(G,B)$, where $b_1,\dots,b_m\in B$ and $\length_B(g^{-1})=m$, showing that the reverse inequality $d\ge m$ holds.   
\end{proof} 

\section{Characterization of directed geodesic atoms}\label{sec:characterization}

The results of Section~\ref{sec:geodesic atoms} show that if we are given the set 
$\mathcal{GA}(G)$ of directed geodesic atoms for $G$, then we can read off generating systems of $G$, diameters of Cayley digraphs of $G$, and we can give upper and lower bounds for the directed Cayley 
diameter of $G$. This motivates our interest in the 
following characterization of the directed geodesic atoms in the set of all atoms $\mathcal{A}(G)$:  

\begin{theorem} \label{thm:equivalent}
The following conditions are equivalent for an element $S\in \mathcal{B}(G)$: 
\begin{itemize} 
\item[(i)] $S$ is not a directed geodesic atom. 
\item[(ii)] For all $g\in \mathrm{supp}(S)$ 
there exists a $T\in \mathcal{GA}(G)$ with $|T|<|S|$,  
such that $g\in \mathrm{supp}(T)$ and 
$\mathrm{supp}(T\bdot g^{[-1]})\subseteq \mathrm{supp}(S\bdot g^{[-1]})$. 
\item[(iii)] For all $g\in \mathrm{supp}(S)$ 
there exists a $T\in \mathcal{A}(G)$ with $|T|<|S|$,  
such that $g\in \mathrm{supp}(T)$ and 
$\mathrm{supp}(T\bdot g^{[-1]})\subseteq \mathrm{supp}(S\bdot g^{[-1]})$. 
\end{itemize} 
\end{theorem}

\begin{proof} $(i)\Longrightarrow (ii)$: 
Assume $S\in \mathcal{B}(G)\setminus \mathcal{GA}(G)$. 
Pick an arbitrary $g$ from $\mathrm{supp}(S)$. Then we can write 
$S=b_1\bdot \dots \bdot b_d\bdot g$, with  $g^{-1}=b_1\cdots b_d$.  
Since $S\notin \mathcal{GA}(G)$, necessarily we have $\length_B(g^{-1})<d$, where 
$B=\supp(b_1\bdot \dots \bdot b_d)$. 
It follows that setting $k:=\length_B(g^{-1})$, there exist elements 
$c_1,\dots,c_k\in B$ such that $g^{-1}=c_1\cdots c_k$. 
Set $T:=c_1\bdot \dots \bdot c_k\bdot g$. 
Then $T\in \mathcal{GA}(G)$ by  
Definition~\ref{def:geodesic atom}.
Moreover, $|T|=k+1<d+1=|S|$, $g\in \mathrm{supp}(T)$, and $\mathrm{supp}(T\bdot g^{[-1]})\subseteq 
B=\mathrm{supp}(S\bdot g^{[-1]})$.  

$(ii)\Longrightarrow (iii)$: This is 
trivial, since $\mathcal{GA}(G)\subseteq \mathcal{A}(G)$ by \eqref{eq:GA subset A}. 
 
$(iii)\Longrightarrow (i)$:  
Suppose for contradiction that (iii) holds for $S$ and $S\in \mathcal{GA}(G)$. 
Then we can write $S=b_1\bdot \dots \bdot b_d\bdot b_{d+1}$, where $b_1\cdots b_d b_{d+1}=1_G$, and setting 
$B:=\supp(b_1\bdot \dots \bdot b_d)$, we have 
$\length_B(b_{d+1}^{-1})=d$. 
Now for $g:=b_{d+1}$ we have $\mathrm{supp}(S\bdot g^{[-1]})=B$.  
Take any $T\in \mathcal{A}(G)$ with $g\in \mathrm{supp}(T)$, $\mathrm{supp}(T\bdot g^{[-1]})\subseteq B$ and $|T|<|S|$. 
We have $T=c_1\bdot \dots \bdot c_k\bdot g$ with $c_i\in B$. With an appropriate numbering of the $c_i$, we have $c_1\cdots c_kg=1_G$, so $c_1\cdots c_k=g^{-1}=b_{d+1}^{-1}$. It follows that $k\ge \length_B(b_{d+1}^{-1})=d$, implying in turn that $|T|=k+1\ge d+1=|S|$. This contradiction proves our claim. 
\end{proof}

\begin{remark} 
An algorithm to compute the set $\mathcal{A}(G)$ of atoms in $\mathcal{B}(G)$ is given in \cite[Section 6]{cz-d-sz}. 
Now Theorem~\ref{thm:equivalent} furnishes a simple algorithm to 
build up the set $\mathcal{GA}(G)$ from 
$\mathcal{A}(G)$. Namely, suppose we found all elements in 
$\mathcal{GA}(G)$ with length at most $t$. Then we have to go through all the atoms in $\mathcal{A}(G)$ of length $t+1$, and by 
Theorem~\ref{thm:equivalent} it is easy 
to select those which belong to $\mathcal{GA}(G)$.  
In Section~\ref{sec:algorithm} we 
shall present an algorithm which computes 
directly $\mathcal{GA}(G)$, without a prior computation of $\mathcal{A}(G)$. 
\end{remark} 


\section{Computing \texorpdfstring{$\GD{G}$}{GD(G)} and 
\texorpdfstring{$\ddiamstar{G}$}{dcd*(G)}}\label{sec:algorithm}

Recall the following definition from \cite{cz-d-sz}:
\begin{definition}
    Let $S, T\in\mathcal{F}(G)$.
    The sequence $T$ is called \emph{a splitting} of the sequence $S$  if  $T = S \bdot g^{[-1]} \bdot x \bdot y$ for some $g\in \supp(S)$ and $x,y \in G$ such that $xy=g$. We denote by $S \prec T$ the fact that $T$ is a splitting of $S$. For any subset $\mathcal S \subset \mathcal F(G)$ we set $\gamma(\mathcal S) := \{ T \in \mathcal F(G) \mid S\prec T \text{ for some } S \in \mathcal S  \}$.
\end{definition}

Let us denote simply by $\mathcal{GA}_k$ the set of all directed geodesic atoms of length $k\ge 0$ in $\mathcal{GA}(G)$ and by $\mathcal{GR}_k$ a set of 
$\Aut(G)$-orbit representatives of $\mathcal{GA}_k$ 
(here $\Aut(G)$ stands for the automorphism group of $G$).  
The possibility to build up $\mathcal{GA}_k$ recursively is based on the following lemma: 

\begin{lemma}\label{splitting}
For any $k \ge 1$ we have $\mathcal{GA}_{k+1} \subseteq \gamma(\mathcal{GA}_k)$. 
\end{lemma}
\begin{proof}
All product-one sequences of length one or two (except the sequence $1_G\bdot 1_G$) are directed geodesic atoms. Assume that $k\ge 2$, and take an arbitrary $S\in \mathcal{GA}_{k+1}$.  Then 
\[S= g_1\bdot\dots \bdot g_k\bdot g_{k+1},  
\mbox{ where }g_1\dots g_kg_{k+1}=1_G 
\mbox{ and } 
\length_{\supp(g_1 \bdot \dots \bdot g_k)}(g_{k+1}^{-1})=k.\] 
Set 
\[T:=g_1\bdot \dots\bdot g_{k-1}\bdot h,
\quad \mbox{where}\quad h:=g_kg_{k+1}.\]
 Then obviously we have $g_1\dots g_{k-1}h=1_G$. We claim that 
$T$ is a directed geodesic atom. 
Suppose for contradiction that 
\[h^{-1}=c_1\cdots c_m 
\mbox{ for some }m<k-1 \mbox{ and }c_1,\dots,c_m\in \{g_1,\dots,g_{k-1}\}.\] 
Taking into account $h^{-1}=g_{k+1}^{-1}g_k^{-1}$ we conclude 
\[g_{k+1}^{-1}=c_1\dots c_mg_k,
\quad \mbox{where}\quad c_1,\dots,c_m,g_k\in   
\{g_1,\dots,g_k\}.
\]
 Thus 
 \[\length_{\supp(g_1\bdot \dots \bdot g_k)}(g_{k+1}^{-1})\le m+1<k,\] 
 a contradiction. 
Therefore $T\in \mathcal{GA}_k$ and  
$T\prec S$, hence $S\in \gamma(\mathcal{GA}_k)$. 
Since $S$ above was an arbitrary element of $\mathcal{GA}_{k+1}$, the statement is proved. 
\end{proof}

\begin{remark}\label{happyend}
Observe that by this lemma if $\mathcal{GA}_k = \emptyset$ for some $k \ge 1$ then $\mathcal{GA}_l = \emptyset$  for every $l >k$. So the smallest $k$ such that $\mathcal{GA}_k = \emptyset$ equals  $ \GD{G}+1$.
\end{remark}

\begin{lemma}\label{reps}
For any $k \ge 1$  let $\mathcal{GR}_k$ be a set of $\Aut(G)$-orbit representatives in $\mathcal{GA}_k$. 
Then there is a system of $\Aut(G)$-orbit representatives $\mathcal{GR}_{k+1}$ in $\mathcal{GA}_{k+1}$ such that $\mathcal{GR}_{k+1} \subseteq \gamma(\mathcal{GR}_k)$. 
\end{lemma}

\begin{proof}
Observe that the relation $\prec$ is compatible with the $\Aut(G)$-action in the sense that $S \prec T$ holds if and only if $\alpha(S) \prec \alpha(T)$ for some $\alpha \in \Aut(G)$ if and only if $\alpha(S) \prec \alpha(T)$ for each $\alpha \in \Aut(G)$. 

Now we prove that for any  sequence $T \in \mathcal{GA}_{k+1}$ its $\Aut(G)$-orbit $\Orb(T)$ has a non-empty intersection with $\gamma(\mathcal{GR}_k)$. 
By Lemma~\ref{splitting} there exists a sequence $S\prec T$ so that $S \in \mathcal{GA}_k$. 
As $\mathcal{GR}_k\subseteq  \mathcal{GA}_k$ is a complete set of 
$\Aut(G)$-orbit representatives, there is some $R \in \mathcal{GR}_k$ such that $R =\alpha (S)$ for some $\alpha \in \Aut(G)$.
Hence by the compatibility of $\prec$ it follows that $ R = \alpha(S) \prec \alpha(T) $. As a result $\alpha(T) \in \Orb(T) \cap \gamma(\mathcal{GR}_k)$.
\end{proof}

Given an arbitrary finite group $G$, Algorithm~\ref{geodesiclargedavenportalgorithm} below computes the geodesic large Davenport constant $\GD{G}$. The set of all directed geodesic atoms of $G$ may be obtained as $\mathcal{GA}(G)=\bigcup_k \mathcal{GA}_k$ after the repeat-until loop has terminated (right after line \ref{Line19}).

The algorithm begins with initializing the sets $\mathcal{GA}_1$ and $\mathcal{GR}_1$ with the sequence consisting only of the identity element of the group. In line \ref{Line11} the variable $S'$ will take successively all the values from the set $\gamma(\mathcal{GR}_{k-1})$, therefore it is guaranteed by Lemma \ref{reps} that for $k>1$ all the possible candidate sequences will be tested for being directed geodesic atoms and the set $\mathcal{GA}_k$ is built correctly. In line \ref{Line10} the variable $x$ is prevented from taking the values $1_G$ or $g$, since this would bring in the identity element in the support of $S'$.

The algorithm first checks whether a candidate sequence $S'$ is among the already computed directed geodesic atoms of length $k$ via the condition $S' \notin \mathcal{GA}_k$ from line \ref{Line12}, ensuring that $S'$ is a new directed geodesic atom (in the sense that it is not similar to any of the already collected directed geodesic atoms of length $k$). Checking whether a candidate sequence $S'$ is a directed geodesic atom is performed by the second condition in line \ref{Line12}, based on Theorem~\ref{thm:equivalent} (ii). The predicate $\mathbf{NG}(S',\cup_{i=1}^{k-1}\mathcal{GA}_i, k)$ in line \ref{Line12} is true if and only if the product-one sequence $S'$ of length $k$ is not a directed geodesic atom and may be given in the following way:  
\begin{align*}
\mathbf{NG}(S',\cup_{i=1}^{k-1}\mathcal{GA}_i, k)\iff & \forall g\in \mathrm{supp}(S'), \exists 
T\in \cup_{i=1}^{k-1}\mathcal{GA}_i \text{ such that} \\  
& g\in \mathrm{supp}(T) \text{ and } \mathrm{supp}(T\bdot g^{[-1]})\subseteq \mathrm{supp}(S'\bdot g^{[-1]}).\end{align*}

In line \ref{Line13} the newly found directed geodesic atom $S'$ gets added to the set $\mathcal{GR}_k$ of representatives and the set $\mathcal{GA}_k$ of directed geodesic atoms is completed with the $\Aut(G)$ orbit of $S'$, denoted by $\Orb(S')$ (in line \ref{Line14}). In the last step, the value $k-1$ is returned as the geodesic large Davenport constant of $G$, as explained in Remark \ref{happyend}.

\medskip{}
\begin{algorithm}[H]

	\medskip{}
	
	\SetKw{LogicAnd}{and} \SetKw{LogicOr}{or} \SetKw{LogicNot}{not}
	\SetKwInOut{Input}{input} \SetKwInOut{Output}{output}
	
	\Input{$G$ }
	\Output{$\GD{G}$ }
	
	\BlankLine
	\BlankLine
	
	$\mathcal{GA}_{1}\gets \{1_G\}$\;
	$\mathcal{GR}_{1}\gets \{1_G\}$\;
	$k\gets 1$\;
	
	\Repeat{$\mathcal{GA}_k=\emptyset$}{
    	$k\gets k+1$\;
		$\mathcal{GA}_k\gets\emptyset$\;
		$\mathcal{GR}_k\gets \emptyset$\;
		\ForAll{$S\in\mathcal{GR}_{k-1}$}{
			\ForAll{$g\in \supp(S)$}{
				\ForAll{$x\in G\setminus \{1_G, g\}$\label{Line10}}{
                    $S'\gets S\bdot g^{[-1]} \bdot x \bdot (x^{-1}g)$\label{Line11}\;
					\If{$S' \notin \mathcal{GA}_k$ \LogicAnd \LogicNot $\mathbf{NG}(S',\cup_{i=1}^{k-1}\mathcal{GA}_i, k)$\label{Line12}}{
						$\mathcal{GR}_{k}\gets \mathcal{GR}_{k}\cup\{S'\}$\label{Line13}\;
						$\mathcal{GA}_{k}\gets \mathcal{GA}_{k}\cup \Orb(S')$\label{Line14}\;
					}
				}
			}
		}
	}\label{Line19}
	
	\Return $k-1$\label{Line20}\;
	
	\BlankLine
	
\caption{GeodesicLargeDavenport($G$)} \label{geodesiclargedavenportalgorithm}
\end{algorithm}
\medskip

One could perform a  slight modification of Algorithm~\ref{geodesiclargedavenportalgorithm} in Line \ref{Line20} to return also the sequence of sets $(\mathcal{GR}_i)_{1\leq i\leq \GD{G}}$. With this information, we can devise Algorithm~\ref{directedcayleydiameteralgorithm}, which, given an arbitrary finite group $G$,
computes $\ddiamstar{G}$, yielding a lower bound for the directed Cayley diameter of $G$ by Lemma \ref{lemma:cayley diam from geodesic atoms} (i):

\medskip{}
\begin{algorithm}[H]

	\medskip{}
	
	\SetKw{LogicAnd}{and} \SetKw{LogicOr}{or} \SetKw{LogicNot}{not}
	\SetKwInOut{Input}{input} \SetKwInOut{Output}{output}
 \SetKw{Break}{break}
	
	\Input{$G$ }
	\Output{$\ddiamstar{G}$ }
	
	\BlankLine
	\BlankLine
	
	$\GD{G},\ (\mathcal{GR}_i)_{1\leq i\leq \GD{G}}\gets \mathrm{GeodesicLargeDavenport}(G)$\label{line1}\;
	$k\gets \GD{G} + 1$\label{line2}\;
	$dcdstar\_found\gets\textsf{false}$\label{line3}\;
	\Repeat{$dcdstar\_found$}{
    $k\gets k-1$\;
		\ForAll{$S\in \mathcal{GR}_k$}{
			\If{$G=\langle\supp(S)\rangle$\label{line7}}{
				$dcdstar\_found\gets\textsf{true}$\;
             \Break\;
			}
		}
	}
	
	\Return $k-1$\label{line13}\;
	
	\BlankLine
	
\caption{DirectedCayleyDiameterLowerBound($G$)} \label{directedcayleydiameteralgorithm}
\end{algorithm}
\medskip

The algorithm begins by calling (the slightly modified) $\mathrm{GeodesicLargeDavenport}(G)$, which in turn provides the value for $\GD{G}$ and all the sets $\mathcal{GR}_k$ for $1\leq k\leq \GD{G}$ (in line \ref{line1}). Then a simple implementation of  Definition \ref{def:GD(G)} (ii) follows. 
\begin{remark}
The condition in line \ref{line7} can be implemented in GAP (see \cite{GAP4}) using the $\mathsf{Group(gens)}$ command, where $\mathsf{gens}$ is a list of generators. Another possibility would be to use the sets $\mathcal{GA}_k$ to decide whether $\supp(S)$ generates $G$ 
by Proposition~\ref{prop:generators}. 
\end{remark}

\section{Results of computer calculations} \label{sec:calculations} 

In the table below we give 
$\ddiamstar{G}$, $\d{G}$, $\GD{G}$, 
$\D{G}$ for all non-abelian groups of order at most $42$ and the alternating group of degree $5$. In addition, we also give $\ddiam{G}$ for all these groups.  
In several cases it happens that $\ddiamstar{G}=\ddiam{G}=\GD{G}-1$ 
(see Corollary~\ref{prop:monotone cayley diameter}). In all other cases we used GAP to compute directly the value of the directed Cayley diameter, by iterating over all possible irredundant generating systems of the group (up to automorphism) and computing the diameter of the associated directed Cayley graph using standard GAP functions available in the ``Digraphs'' and ``GRAPE'' packages.  
Moreover, in the last row of the table we give 
$\ddiam{G}$ and $\d{G}$ for the smallest non-abelian simple group $G=A_5$ (the alternating group of degree $5$). 

The first column contains the identifier of $G$ in the Small Groups Library of GAP (see \cite{GAP4}); that is, a pair of the form \texttt{[order, i]}, where the GAP command $\mathsf{SmallGroup(id)}$  returns the $i$-th group of order $order$ in the catalogue.

The second column contains the ``structure description'' string of the group, as returned by the $\mathsf{StructureDescription(G)}$ command.

\renewcommand*{\arraystretch}{1.2}
\begin{longtable}[c]{c|c|c|c|c|c|c}
id & Struct. desc. & $\ddiamstar{G}$ & $\ddiam{G}$ & $\d{G}$ & $\GD{G}$ & $\D{G}$\tabularnewline
\hline 
\texttt{{[}\ 6,\,\ 1{]}} & S3 & 3 & 3 & 3 & 4 & 6\tabularnewline
\rowcolor{TableGray}
\texttt{{[}\ 8,\,\ 3{]}} & D8 & 4 & 4 & 4 & 5 & 6\tabularnewline
\texttt{{[}\ 8,\,\ 4{]}} & Q8 & 3 & 3 & 4 & 4 & 6\tabularnewline
\rowcolor{TableGray}
\texttt{{[}10,\,\ 1{]}} & D10 & 5 & 5 & 5 & 6 & 10\tabularnewline
\texttt{{[}12,\,\ 1{]}} & C3 : C4 & 5 & 5 & 6 & 6 & 9\tabularnewline
\rowcolor{TableGray}
\texttt{{[}12,\,\ 3{]}} & A4 & 4 & 4 & 4 & 5 & 7\tabularnewline
\texttt{{[}12,\,\ 4{]}} & D12 & 6 & 6 & 6 & 7 & 9\tabularnewline
\rowcolor{TableGray}
\texttt{{[}14,\,\ 1{]}} & D14 & 7 & 7 & 7 & 8 & 14\tabularnewline
\texttt{{[}16,\,\ 3{]}} & (C4 x C2) : C2 & 6 & 6 & 5 & 7 & 7\tabularnewline
\rowcolor{TableGray}
\texttt{{[}16,\,\ 4{]}} & C4 : C4 & 6 & 6 & 6 & 7 & 8\tabularnewline
\texttt{{[}16,\,\ 6{]}} & C8 : C2 & 6 & 6 & 8 & 8 & 10\tabularnewline
\rowcolor{TableGray}
\texttt{{[}16,\,\ 7{]}} & D16 & 8 & 8 & 8 & 9 & 12\tabularnewline
\texttt{{[}16,\,\ 8{]}} & QD16 & 5 & 5 & 8 & 8 & 12\tabularnewline
\rowcolor{TableGray}
\texttt{{[}16,\,\ 9{]}} & Q16 & 5 & 5 & 8 & 8 & 12\tabularnewline
\texttt{{[}16,\,11{]}} & C2 x D8 & 5 & 5 & 5 & 6 & 7\tabularnewline
\rowcolor{TableGray}
\texttt{{[}16,\,12{]}} & C2 x Q8 & 4 & 4 & 5 & 5 & 7\tabularnewline
\texttt{{[}16,\,13{]}} & (C4 x C2) : C2 & 4 & 4 & 5 & 5 & 7\tabularnewline
\rowcolor{TableGray}
\texttt{{[}18,\,\ 1{]}} & D18 & 9 & 9 & 9 & 10 & 18\tabularnewline
\texttt{{[}18,\,\ 3{]}} & C3 x S3 & 7 & 7 & 7 & 8 & 10\tabularnewline
\rowcolor{TableGray}
\texttt{{[}18,\,\ 4{]}} & (C3 x C3) : C2 & 4 & 4 & 5 & 5 & 10\tabularnewline
\texttt{{[}20,\,\ 1{]}} & C5 : C4 & 7 & 7 & 10 & 10 & 15\tabularnewline
\rowcolor{TableGray}
\texttt{{[}20,\,\ 3{]}} & C5 : C4 & 6 & 6 & 7 & 7 & 10\tabularnewline
\texttt{{[}20,\,\ 4{]}} & D20 & 10 & 10 & 10 & 11 & 15\tabularnewline
\rowcolor{TableGray}
\texttt{{[}21,\,\ 1{]}} & C7 : C3 & 5 & 5 & 8 & 7 & 14\tabularnewline
\texttt{{[}22,\,\ 1{]}} & D22 & 11 & 11 & 11 & 12 & 22\tabularnewline
\rowcolor{TableGray}
\texttt{{[}24,\,\ 1{]}} & C3 : C8 & 9 & 9 & 12 & 12 & 15\tabularnewline
\texttt{{[}24,\,\ 3{]}} & SL(2,3) & 6 & 6 & 7 & 7 & 13\tabularnewline
\rowcolor{TableGray}
\texttt{{[}24,\,\ 4{]}} & C3 : Q8 & 7 & 7 & 12 & 12 & 18\tabularnewline
\texttt{{[}24,\,\ 5{]}} & C4 x S3 & 8 & 8 & 12 & 12 & 15\tabularnewline
\rowcolor{TableGray}
\texttt{{[}24,\,\ 6{]}} & D24 & 12 & 12 & 12 & 13 & 18\tabularnewline
\texttt{{[}24,\,\ 7{]}} & C2 x (C3 : C4) & 8 & 8 & 8 & 9 & 11\tabularnewline
\rowcolor{TableGray}
\texttt{{[}24,\,\ 8{]}} & (C6 x C2) : C2 & 7 & 7 & 7 & 8 & 14\tabularnewline
\texttt{{[}24,\,10{]}} & C3 x D8 & 8 & 8 & 12 & 12 & 14\tabularnewline
\rowcolor{TableGray}
\texttt{{[}24,\,11{]}} & C3 x Q8 & 8 & 8 & 12 & 12 & 14\tabularnewline
\texttt{{[}24,\,12{]}} & S4 & 7 & 7 & 6 & 8 & 12\tabularnewline
\rowcolor{TableGray}
\texttt{{[}24,\,13{]}} & C2 x A4 & 7 & 7 & 7 & 8 & 10\tabularnewline
\texttt{{[}24,\,14{]}} & C2 x C2 x S3 & 7 & 7 & 7 & 8 & 10\tabularnewline
\rowcolor{TableGray}
\texttt{{[}26,\,\ 1{]}} & D26 & 13 & 13 & 13 & 14 & 26\tabularnewline
\texttt{{[}27,\,\ 3{]}} & (C3 x C3) : C3 & 6 & 6 & 6 & 7 & 8\tabularnewline
\rowcolor{TableGray}
\texttt{{[}27,\,\ 4{]}} & C9 : C3 & 6 & 6 & 10 & 9 & 12\tabularnewline
\texttt{{[}28,\,\ 1{]}} & C7 : C4 & 9 & 9 & 14 & 14 & 21\tabularnewline
\rowcolor{TableGray}
\texttt{{[}28,\,\ 3{]}} & D28 & 14 & 14 & 14 & 15 & 21\tabularnewline
\texttt{{[}30,\,\ 1{]}} & C5 x S3 & 11 & 11 & 15 & 15 & 18\tabularnewline
\rowcolor{TableGray}
\texttt{{[}30,\,\ 2{]}} & C3 x D10 & 9 & 9 & 15 & 15 & 20\tabularnewline
\texttt{{[}30,\,\ 3{]}} & D30 & 15 & 15 & 15 & 16 & 30\tabularnewline
\rowcolor{TableGray}
\texttt{{[}32,\,\ 2{]}} & (C4 x C2) : C4 & 8 & 8 & 7 & 9 & 9\tabularnewline
\texttt{{[}32,\,\ 4{]}} & C8 : C4 & 7 & 7 & 10 & 9 & 12\tabularnewline
\rowcolor{TableGray}
\texttt{{[}32,\,\ 5{]}} & (C8 x C2) : C2 & 10 & 10 & 9 & 11 & 11\tabularnewline
\texttt{{[}32,\,\ 6{]}} & (C2 x C2 x C2) : C4 & 8 & 8 & 7 & 9 & 10\tabularnewline
\rowcolor{TableGray}
\texttt{{[}32,\,\ 7{]}} & (C8 : C2) : C2 & 7 & 7 & 9 & 8 & 12\tabularnewline
\texttt{{[}32,\,\ 8{]}} & (C2 x C2) . (C4 x C2) & 7 & 7 & 9 & 8 & 12\tabularnewline
\rowcolor{TableGray}
\texttt{{[}32,\,\ 9{]}} & (C8 x C2) : C2 & 10 & 10 & 9 & 11 & 13\tabularnewline
\texttt{{[}32,\,10{]}} & Q8 : C4 & 7 & 7 & 9 & 9 & 13\tabularnewline
\rowcolor{TableGray}
\texttt{{[}32,\,11{]}} & (C4 x C4) : C2 & 8 & 8 & 9 & 9 & 14\tabularnewline
\texttt{{[}32,\,12{]}} & C4 : C8 & 10 & 10 & 10 & 11 & 12\tabularnewline
\rowcolor{TableGray}
\texttt{{[}32,\,13{]}} & C8 : C4 & 7 & 7 & 10 & 9 & 14\tabularnewline
\texttt{{[}32,\,14{]}} & C8 : C4 & 10 & 10 & 10 & 11 & 14\tabularnewline
\rowcolor{TableGray}
\texttt{{[}32,\,15{]}} & C4 . D8 = C4 . (C4 x C2) & 8 & 8 & 10 & 9 & 14\tabularnewline
\texttt{{[}32,\,17{]}} & C16 : C2 & 10 & 10 & 16 & 16 & 18\tabularnewline
\rowcolor{TableGray}
\texttt{{[}32,\,18{]}} & D32 & 16 & 16 & 16 & 17 & 24\tabularnewline
\texttt{{[}32,\,19{]}} & QD32 & 9 & 9 & 16 & 16 & 24\tabularnewline
\rowcolor{TableGray}
\texttt{{[}32,\,20{]}} & Q32 & 9 & 9 & 16 & 16 & 24\tabularnewline
\texttt{{[}32,\,22{]}} & C2 x ((C4 x C2) : C2) & 7 & 7 & 6 & 8 & 8\tabularnewline
\rowcolor{TableGray}
\texttt{{[}32,\,23{]}} & C2 x (C4 : C4) & 7 & 7 & 7 & 8 & 9\tabularnewline
\texttt{{[}32,\,24{]}} & (C4 x C4) : C2 & 6 & 6 & 7 & 7 & 9\tabularnewline
\rowcolor{TableGray}
\texttt{{[}32,\,25{]}} & C4 x D8 & 7 & 7 & 7 & 8 & 9\tabularnewline
\texttt{{[}32,\,26{]}} & C4 x Q8 & 6 & 6 & 7 & 7 & 9\tabularnewline
\rowcolor{TableGray}
\texttt{{[}32,\,27{]}} & (C2 x C2 x C2 x C2) : C2 & 6 & 6 & 6 & 7 & 9\tabularnewline
\texttt{{[}32,\,28{]}} & (C4 x C2 x C2) : C2 & 6 & 6 & 7 & 7 & 10\tabularnewline
\rowcolor{TableGray}
\texttt{{[}32,\,29{]}} & (C2 x Q8) : C2 & 6 & 6 & 7 & 7 & 10\tabularnewline
\texttt{{[}32,\,30{]}} & (C4 x C2 x C2) : C2 & 6 & 6 & 7 & 7 & 10\tabularnewline
\rowcolor{TableGray}
\texttt{{[}32,\,31{]}} & (C4 x C4) : C2 & 6 & 6 & 7 & 7 & 10\tabularnewline
\texttt{{[}32,\,32{]}} & (C2 x C2) . (C2 x C2 x C2) & 5 & 5 & 7 & 7 & 10\tabularnewline
\rowcolor{TableGray}
\texttt{{[}32,\,33{]}} & (C4 x C4) : C2 & 5 & 5 & 7 & 7 & 10\tabularnewline
\texttt{{[}32,\,34{]}} & (C4 x C4) : C2 & 6 & 6 & 7 & 7 & 10\tabularnewline
\rowcolor{TableGray}
\texttt{{[}32,\,35{]}} & C4 : Q8 & 6 & 6 & 7 & 7 & 10\tabularnewline
\texttt{{[}32,\,37{]}} & C2 x (C8 : C2) & 7 & 7 & 9 & 9 & 11\tabularnewline
\rowcolor{TableGray}
\texttt{{[}32,\,38{]}} & (C8 x C2) : C2 & 7 & 7 & 9 & 9 & 11\tabularnewline
\texttt{{[}32,\,39{]}} & C2 x D16 & 9 & 9 & 9 & 10 & 13\tabularnewline
\rowcolor{TableGray}
\texttt{{[}32,\,40{]}} & C2 x QD16 & 6 & 6 & 9 & 9 & 13\tabularnewline
\texttt{{[}32,\,41{]}} & C2 x Q16 & 6 & 6 & 9 & 9 & 13\tabularnewline
\rowcolor{TableGray}
\texttt{{[}32,\,42{]}} & (C8 x C2) : C2 & 6 & 6 & 9 & 9 & 13\tabularnewline
\texttt{{[}32,\,43{]}} & C8 : (C2 x C2) & 6 & 6 & 9 & 9 & 12\tabularnewline
\rowcolor{TableGray}
\texttt{{[}32,\,44{]}} & (C2 x Q8) : C2 & 5 & 5 & 9 & 8 & 12\tabularnewline
\texttt{{[}32,\,46{]}} & C2 x C2 x D8 & 6 & 6 & 6 & 7 & 8\tabularnewline
\rowcolor{TableGray}
\texttt{{[}32,\,47{]}} & C2 x C2 x Q8 & 5 & 5 & 6 & 6 & 8\tabularnewline
\texttt{{[}32,\,48{]}} & C2 x ((C4 x C2) : C2) & 5 & 5 & 6 & 6 & 8\tabularnewline
\rowcolor{TableGray}
\texttt{{[}32,\,49{]}} & (C2 x C2 x C2) : (C2 x C2) & 4 & 4 & 6 & 6 & 8\tabularnewline
\texttt{{[}32,\,50{]}} & (C2 x Q8) : C2 & 4 & 4 & 6 & 5 & 8\tabularnewline
\rowcolor{TableGray}
\texttt{{[}34,\,\ 1{]}} & D34 & 17 & 17 & 17 & 18 & 34\tabularnewline
\texttt{{[}36,\,\ 1{]}} & C9 : C4 & 11 & 11 & 18 & 18 & 27\tabularnewline
\rowcolor{TableGray}
\texttt{{[}36,\,\ 3{]}} & (C2 x C2) : C9 & 10 & 10 & 10 & 11 & 13\tabularnewline
\texttt{{[}36,\,\ 4{]}} & D36 & 18 & 18 & 18 & 19 & 27\tabularnewline
\rowcolor{TableGray}
\texttt{{[}36,\,\ 6{]}} & C3 x (C3 : C4) & 13 & 13 & 13 & 14 & 16\tabularnewline
\texttt{{[}36,\,\ 7{]}} & (C3 x C3) : C4 & 6 & 6 & 8 & 8 & 13\tabularnewline
\rowcolor{TableGray}
\texttt{{[}36,\,\ 9{]}} & (C3 x C3) : C4 & 8 & 8 & 7 & 9 & 12\tabularnewline
\texttt{{[}36,\,10{]}} & S3 x S3 & 8 & 8 & 8 & 9 & 14\tabularnewline
\rowcolor{TableGray}
\texttt{{[}36,\,11{]}} & C3 x A4 & 8 & 8 & 8 & 9 & 11\tabularnewline
\texttt{{[}36,\,12{]}} & C6 x S3 & 10 & 10 & 10 & 11 & 13\tabularnewline
\rowcolor{TableGray}
\texttt{{[}36,\,13{]}} & C2 x ((C3 x C3) : C2) & 7 & 7 & 8 & 8 & 13\tabularnewline
\texttt{{[}38,\,\ 1{]}} & D38 & 19 & 19 & 19 & 20 & 38\tabularnewline
\rowcolor{TableGray}
\texttt{{[}39,\,\ 1{]}} & C13 : C3 & 7 & 7 & 14 & 13 & 26\tabularnewline
\texttt{{[}40,\,\ 1{]}} & C5 : C8 & 11 & 11 & 20 & 20 & 25\tabularnewline
\rowcolor{TableGray}
\texttt{{[}40,\,\ 3{]}} & C5 : C8 & 10 & 10 & 12 & 11 & 15\tabularnewline
\texttt{{[}40,\,\ 4{]}} & C5 : Q8 & 11 & 11 & 20 & 20 & 30\tabularnewline
\rowcolor{TableGray}
\texttt{{[}40,\,\ 5{]}} & C4 x D10 & 12 & 12 & 20 & 20 & 25\tabularnewline
\texttt{{[}40,\,\ 6{]}} & D40 & 20 & 20 & 20 & 21 & 30\tabularnewline
\rowcolor{TableGray}
\texttt{{[}40,\,\ 7{]}} & C2 x (C5 : C4) & 12 & 12 & 12 & 13 & 17\tabularnewline
\texttt{{[}40,\,\ 8{]}} & (C10 x C2) : C2 & 11 & 11 & 11 & 12 & 22\tabularnewline
\rowcolor{TableGray}
\texttt{{[}40,\,10{]}} & C5 x D8 & 12 & 12 & 20 & 20 & 22\tabularnewline
\texttt{{[}40,\,11{]}} & C5 x Q8 & 12 & 12 & 20 & 20 & 22\tabularnewline
\rowcolor{TableGray}
\texttt{{[}40,\,12{]}} & C2 x (C5 : C4) & 9 & 9 & 12 & 11 & 15\tabularnewline
\texttt{{[}40,\,13{]}} & C2 x C2 x D10 & 11 & 11 & 11 & 12 & 16\tabularnewline
\rowcolor{TableGray}
\texttt{{[}42,\,\ 1{]}} & C7 : C6 & 8 & 8 & 11 & 9 & 14\tabularnewline
\texttt{{[}42,\,\ 2{]}} & C2 x (C7 : C3) & 8 & 8 & 15 & 14 & 21\tabularnewline
\rowcolor{TableGray}
\texttt{{[}42,\,\ 3{]}} & C7 x S3 & 15 & 15 & 21 & 21 & 24\tabularnewline
\texttt{{[}42,\,\ 4{]}} & C3 x D14 & 11 & 11 & 21 & 21 & 28\tabularnewline
\rowcolor{TableGray}
\texttt{{[}42,\,\ 5{]}} & D42 & 21 & 21 & 21 & 22 & 42\tabularnewline
\texttt{{[}60,\,\ 5{]}} & A5 & 12 & 12 & 8 & 13 & -\tabularnewline

\end{longtable}

\begin{example} 
The table above shows that 
for $3\le n\le 21$ the dihedral group $D_{2n}$ has directed Cayley diameter $n$. It is quite easy to see 
that $\ddiam{D_{2n}}\ge n$ for all $n\ge 3$. Indeed, $D_{2n}$ has the presentation with generators and relations as 
$D_{2n}=\langle a,b\mid a^2=b^2=(ab)^n=1\rangle$. 
Now when $n$ is even, then 
$\length_{\{a,b\}}((ab)^{n/2})=n$, 
and when $n$ is odd, 
then 
$\length_{\{a,b\}}((ab)^{(n-1)/2}a)=n$. 
\end{example}

\begin{remark}\label{remark:monotonity} The quantities $\D{G}$, 
$\d{G}$, $\GD{G}$ are obviously monotone in  the sense that for any subgroup $H$ of $G$ we have 
$\D{H}\le \D{G}$, $\d{H}\le \d{G}$, 
$\GD{H}\le \GD{G}$. 
The quantity $\ddiam{G}$ is also monotone for abelian groups $G$ by 
\eqref{eq:d*(G)} and Theorem~\ref{thm:diam=d*}. 
On the other hand, the quantity $\ddiam{G}$ is not monotone for all finite groups $G$. 
This follows from some known cases of  Babai's Conjecture from \cite{babai-seress} mentioned in the Introduction,  
because a non-abelian simple group 
may contain a ``large'' cyclic subgroup,  having ``large'' directed Cayley diameter. 
The above table provides small examples, e.g. 
if $G$ is the order $16$ group with GAP identifier $\texttt{[16,6]}$, then $\ddiam{G}=6$, and $G$ has an order $8$ cyclic subgroup $H$, for which $\ddiam{H}=7$.  
\end{remark}

\begin{remark}
The equality $\ddiam{G}=\ddiamstar{G}$ 
holds for all abelian groups $G$ by 
Theorem~\ref{thm:diam=d*}, and for all 
non-abelian groups of order at most $42$, as one can observe inspecting the above table.  We do not know whether the equality 
$\ddiam{G}=\ddiamstar{G}$ holds for all finite groups $G$. 
\end{remark}

\begin{remark}
    The computer calculations were performed on a server with \qty{3.6}{\GHz} processors and \qty{1}{\tera\byte} of RAM memory, running single-threaded implementations in GAP of Algorithm~\ref{geodesiclargedavenportalgorithm}, Algorithm~\ref{directedcayleydiameteralgorithm}, respectively the algorithms computing the small and large Davenport constants given in \cite{cz-d-sz}. 
    
    The results for the groups of order at most $32$ have been obtained in $30$ hours of computing time. The results for the remaining groups up to order $39$ have been obtained in about two weeks' time. Computations for groups of order $40$ and $42$ took about three months, the group with GAP identifier $[40, 8]$ being the most time-consuming (computation for this group alone took almost 64 days).
    
    Computing the large Davenport constant (using Algorithm 2 from \cite{cz-d-sz}) was the most resource-consuming (in terms of computing-time and memory). Unfortunately, the algorithm did not finish in time to provide the value for $\D{A_5}$. We are working on a parallel implementation and further optimizations to compute all these constants for groups of order above $42$.
\end{remark}


\section{Some known results on the Davenport constants} \label{sec:survey}

The small Davenport constant for non-abelian groups was introduced in \cite{olson-white}, where it was shown that the inequality $\d{G}\le\frac 12|G|$ holds for any non-cyclic group $G$, with equality when $G$ has a cyclic index $2$ subgroup. 
The large Davenport constant was introduced in \cite{geroldinger-grynkiewicz}, where it was proved that 
if $G$ has a cyclic index $2$ subgroup, then 
$\D{G}=\d{G}+|G'|$;  
here $G'$ stands for the commutator subgroup of $G$. 
Bounds on $\d{G}$ in terms of $|G|$ and the smallest prime divisor $p$ of $|G|$ were given in a sequence of papers (see \cite{grynkiewicz}, \cite{gao-li-peng}). 
To the best of our knowledge, the strongest published bound of this type is given in \cite{qu-li-teeuwsen}, 
where it is proved that 
$\d{G}\le |G|/p+p-2$. 


\end{document}